\RequirePackage{etoolbox}
\csdef{input@path}{%
 {sty/},
 {img/},
 {bib/}
}
\documentclass[numbers,compress]{vmsta}

\volume{0}
\issue{0}
\pubyear{2022}
\articletype{research-article}

\usepackage{amsmath}
\usepackage{amssymb}
\usepackage{amsthm}
\usepackage{bbm}
\usepackage{MnSymbol}
\usepackage{xcolor}
\usepackage[T2A]{fontenc}
\usepackage[utf8]{inputenc}
\usepackage{enumerate}

\newtheorem{thm}{Theorem}
\newtheorem{proposition}{Proposition}
\newtheorem{lemma}{Lemma}

\DeclareMathOperator{\1}{\mathbbm{1}}

\newcommand{\mmp}{\mathbb{P}}
\newcommand{\mn}{\mathbb{N}}
\newcommand{\me}{\mathbb{E}}
\newcommand{\eee}{{\rm e}}
\newcommand{\mr}{\mathbb{R}}

\begin{document}

\begin{frontmatter}

\pretitle{Research Article}

\title{Laws of the iterated logarithm for iterated perturbed random walks}

\author[]{\inits{O.}\fnms{Oksana}~\snm{Braganets}\thanksref{}\ead[label=]{oksanabraganets@knu.ua}\orcid{0009-0004-8348-545X}}
\address[]{\institution{Faculty of Computer Science and Cybernetics, Taras Shevchenko National University of Kyiv}, \cny{Ukraine}}


\begin{abstract}
\noindent
Let $(\xi_k, \eta_k)_{k\geq 1}$ be independent identically distributed random vectors with arbitrarily
dependent positive components and $T_k:=\xi_1+\ldots+\xi_{k-1}+\eta_k$ for $k\in\mn$. We call the random sequence $(T_k)_{k\geq 1}$ a (globally) perturbed random walk.
Consider a general branching process generated by $(T_k)_{k\geq 1}$ and let $Y_j(t)$ denote the number of the $j$th generation individuals with birth times $\leq t$. Assuming that ${\rm Var}\,\xi_1\in (0,\infty)$ and allowing the distribution of $\eta_1$ to be arbitrary, we prove a law of the iterated logarithm (LIL) for $Y_j(t)$. In particular, a LIL for the counting process of $(T_k)_{k\geq 1}$ is obtained. The latter result was previously established in the article Iksanov, Jedidi and Bouzeffour (2017) under the additional assumption that $\me \eta_1^a<\infty$ for some $a>0$. In this paper, we show that the aforementioned additional assumption is not needed.
\end{abstract}



\end{frontmatter}

\section{Introduction and main results}

Let $(\xi_k, \eta_k)_{k\geq 1}$ be independent copies of a random vector $(\xi, \eta)$ with positive arbitrarily
dependent components. Put $$S_0:=0,\quad S_k:=\xi_1+\ldots+\xi_k,\quad k\in\mn:=\{1,2,\ldots\}$$ and then
$$T_k:=S_{k-1}+\eta_k,\quad k\in\mn.$$ The random sequences $S:=(S_k)_{k\geq 0}$ and $T:= (T_k)_{k\geq 1}$ are known in the literature as the {\it standard random walk} and a {\it (globally) perturbed
random walk}. A survey of various
results for the so defined perturbed random walks can be found in the book \cite{Iksanov:2016}.

Put
\begin{equation*}
Y(t)~:=~\sum_{k\geq 0}\1_{\{T_k\leq t\}}, \quad t\geq 0.
\end{equation*}
A law of the iterated logarithm (LIL) for $Y(t)$, properly normalized and centered, as $t\to\infty$ {\it along integers} was proved in Proposition 2.3 of \cite{Iksanov+Jedidi+Bouzzefour:2017} under the assumption that $\me \eta^a<\infty$ for some $a>0$. We start by showing that the assumption can be dispensed with and also that the LIL holds as $t\to\infty$ {\it along reals}, thereby obtaining an ultimate version of the LIL for $Y(t)$.
For a family $(x_t)$ of real numbers denote by $C((x_t))$ the set of its limit points.
\begin{thm}\label{imp2}
Assume that $\sigma^2:={\rm Var}\,\xi\in (0,\infty)$. Then $$C\bigg(\bigg(\frac{Y(t)-\mu^{-1}\int_0^t \mmp\{\eta\leq y\}{\rm d}y}{(2\sigma^2\mu^{-3}t\log\log t)^{1/2}}\,:\,t > \eee \bigg)\bigg)=[-1,1]\quad\text{{\rm
a.s.}}, $$ where $\mu:=\me \xi<\infty$.
\end{thm}

Next, we consider a general branching process generated by the random sequence $(T_k)_{k\geq 1}$. Thus, the random variables $T_1, T_2, \ldots$ are interpreted as the birth times of the first generation individuals. The first generation produces the second generation. The shifts of birth times of the second generation individuals with respect to their mothers’ birth times are distributed according to copies of $T$, and for different mothers these copies are independent. The second generation produces the third one, and so on.

Let $Y_j(t)$ be the number of the $j$th generation individuals with birth times $\leq t$. Following \cite{Bohun etal:2022}, we call the sequence of processes $((Y_j(t))_{t\geq 0})_{j\geq 2}$ an iterated perturbed random walk. Note that, for $t \geq 0$,  $Y_1(t) = Y(t)$ and the following decomposition holds
\begin{equation}\label{eq:ndecomp}
Y_j(t)  = \sum_{r\geq 1} Y_{j-1}^{(r)} (t-T_r) \1_{\{T_r \leq t\}}, \quad j\geq 2,
\end{equation}
where $Y_{j-1}^{(r)}(t)$ is the number of  the $j$th generation individuals who are descendants of the
first generation individual with birth time $T_r$. Put $V(t) := V_1(t) = \me Y(t)$ for $t \geq 0$. Passing in \eqref{eq:ndecomp} to expectations we infer, for $j \geq 2$ and $t \geq 0$,
\begin{equation}\label{eq:vdecomp}
V_j(t) = (V_{j-1}\ast V)(t) = \int_{[0,\,t]} V_{j-1} (t-y){\rm d} V(y).
\end{equation}

The iterated perturbed random walks are interesting objects on their own, see \cite{Iksanov+Marynych+Rashytov:2022, Iksanov+Rashytov+Samoilenko:2023}. Also, these are the main auxiliary tool in investigations of nested infinite occupancy schemes in random environment. Details can be found in the papers 
\cite{Braganets+Iksanov:2023, Braganets+Iksanov:2025, Buraczewski+Dovgay+Iksanov:2020, Iksanov+Marynych+Samoilenko:2022}. 
Attention was also paid to iterated standard random walks, which are a rather particular instance of the iterated perturbed random walks which corresponds to $\eta=\xi$. 
A LIL for the iterated standard random walks was recently proved in \cite{Iksanov+Kabluchko+Kotelnikova:2025}. Continuing this line of investigation we formulate and prove a LIL for $Y_j(t)$, properly normalized and centered, as $t\to\infty$. 
\begin{thm}\label{main}
Assume that $\sigma^2={\rm Var}\,\xi\in (0,\infty)$. Then, for $j\geq 2$,
\begin{equation}\label{eq:main}
C\bigg(\bigg(\frac{ Y_j(t)-V_j(t)}{  (2((2j-1)(j-1)!)^{-1} \sigma^2 \mu^{-2j-1} t^{2j-1}\log\log t)^{1/2} }\,:\,t > \eee\bigg)\bigg)=[-1,1]\quad\text{{\rm
a.s.}} ,
\end{equation}
where $\mu=\me \xi<\infty$.
\end{thm}

Although the beginning of our proof of Theorem \ref{main} is similar to that of Theorem 1.1 in \cite{Iksanov+Kabluchko+Kotelnikova:2025}, the subsequent 
technical details are essentially different. The main difficulty is that the distribution of $\eta$ is arbitrary. Imposing a moment assumption on the distribution of 
$\eta$ would greatly simplify an argument. 

The remainder of the paper is structured as follows. After proving Theorem \ref{imp2} in Section \ref{sect:lil}, we give a number of auxiliary
results in Section \ref{sect:aux} and then prove Theorem \ref{main} in Section 4.

\section{Proof of Theorem \ref{imp2}}\label{sect:lil}

We shall denote by $n$ an integer argument and by $t$ a real argument. For $t\in\mr$, put $F(t):=\mmp\{\eta\leq t\}$ and
\begin{equation}\label{def:nu}
\nu(t):=\sum_{k\geq 0}\1_{\{S_k\leq t\}},
\end{equation}
 and observe that $F(t)=0$ and $\nu(t)=0$ for $t<0$. For $t>\eee$, 
 write
\begin{multline*}
Y(t)-\mu^{-1}\int_0^t F(y){\rm d}y=Y(t)-\int_{[0,\,t]}F(t-y){\rm d}\nu(y)\\+\int_{[0,\,t]}F(t-y){\rm d}(\nu(y)-\mu^{-1}y)=:X(t)+Z(t)
\end{multline*}
and put $a(t):=(2\sigma^2\mu^{-3}t\log\log t)^{1/2}$. It is shown in the proof of Proposition 2.3 in \cite{Iksanov+Jedidi+Bouzzefour:2017} that
\begin{equation}\label{inter1}
C\big(\big(Z(n)/a(n)\ : \ n\geq 3\big)\big)=[-1,1]\quad\text{a.s.}
\end{equation}
This result holds irrespective of whether $\me \eta^a<\infty$ for some $a>0$ or $\me \eta^a=\infty$ for all $a>0$. We intend to show that \eqref{inter1} entails
\begin{equation}\label{inter11}
C\big(\big(Z(t)/a(t)\ : \ t>\eee 
\big)\big)=[-1,1]\quad\text{a.s.}
\end{equation}
Given $t\geq 4$ 
there exists $n\in\mn$ such that $t\in (n-1, n]$. Hence, by monotonicity, $$\frac{Z(t)}{a(t)}\leq \frac{Z(n)+\mu^{-1}\int_{n-1}^n F(y){\rm d}y}{a(n-1)}\leq \frac{Z(n)+\mu^{-1}}{a(n-1)}\quad\text{a.s.}$$ Analogously, $$\frac{Z(t)}{a(t)}\geq \frac{Z(n-1)-\mu^{-1}}{a(n)}\quad \text{a.s.}$$ We conclude that \eqref{inter11} does indeed hold.

It is known (see the proof of Theorem 3.2 in \cite{Alsmeyer+Iksanov+Marynych:2017}) that $$\lim_{n\to\infty}\,n^{-1/2}\bigg(Y(n)-\int_{[0,\,n]}F(n-y){\rm d}\nu(y)\bigg)=0\quad\text{a.s.}$$ whenever $\me \eta^a<\infty$ for some $a>0$. We note, without going into details that the latter limit relation may fail to hold if $\me \eta^a=\infty$ for all $a>0$. For instance, if $\lim_{t\to\infty}(\log\log t)(1-F(t))=+\infty$, then the limit superior in the last centered formula is equal to $+\infty$ a.s. The proof of Theorem 3.2 in \cite{Alsmeyer+Iksanov+Marynych:2017} operates with power moments and relies heavily upon the assumption $\me \eta^a<\infty$ for some $a>0$. Without such an assumption another argument is needed, which operates with exponential rather than power moments.
In the remainder of the proof we present such an argument, which enables us to prove that
\begin{equation}\label{eq:final}
\lim_{t\to\infty}(X(t)/b(t))=0\quad\text{a.s.},
\end{equation}
thereby completing the proof of the theorem. Here, $b(t):=(t\log\log t)^{1/2}$ for $t>\eee$. 

Fix any $u\neq 0$ and $t>0$. Put $W_0:=1$ and, for $j\in\mn$,
\begin{multline*}
W_j:=\exp\Big(u\sum_{k=0}^{j-1}(\1_{\{\eta_{k+1}+S_k\leq t\}}-F(t-S_k)\1_{\{S_k\leq t\}})\\-(u^2\eee^{|u|}/2)\sum_{k=0}^{j-1}(1-F(t-S_k))\1_{\{S_k\leq t\}}\Big),
\end{multline*}
and denote by $\mathcal{G}_0$ the trivial $\sigma$-algebra and, for $j\in\mn$, by $\mathcal{G}_j$ the $\sigma$-algebra generated by $(\xi_k,\eta_k)_{1\leq k\leq j}$. Observe that the variable $W_j$ is $\mathcal{G}_j$-measurable for $j\in\mn_0:=\mn\cup\{0\}$. Now we prove that $(W_j, \mathcal{G}_j)_{j\geq 0}$ is a positive supermartingale. Indeed, writing $\me_j(\cdot)$ for $\me (\cdot|\mathcal{G}_j)$ and using the inequality $\eee^x\leq 1+x+x^2\eee^{|x|}/2$ for $x\in\mr$ in combination with $$\me_{j-1}\big(\1_{\{\eta_j+S_{j-1}\leq t\}}-F(t-S_{j-1})\1_{\{S_{j-1}\leq t\}}\big)=0\quad\text{a.s.}$$ we infer
\begin{multline*}
\me_{j-1}\exp\big(u(\1_{\{\eta_j+S_{j-1}\leq t\}}-F(t-S_{j-1})\1_{\{S_{j-1}\leq t\}})\big)\\\leq 1+(u^2/2)\me_{j-1}(\1_{\{T_j\leq t\}}-F(t-S_{j-1}))^2 \\ \times \exp(|u(\1_{\{T_j\leq t\}}-F(t-S_{j-1})\1_{\{S_{j-1}\leq t\}})|)\1_{\{S_{j-1}\leq t\}}.
\end{multline*}
In view of $|\1_{\{T_j\leq t\}}-F(t-S_{j-1})\1_{\{S_{j-1}\leq t\}}|\leq 1$ a.s., the right-hand side does not exceed
\begin{multline*}
1+(u^2\eee^{|u|}/2)F(t-S_{j-1})(1-F(t-S_{j-1}))\1_{\{S_{j-1}\leq t\}}\\\leq 1+(u^2\eee^{|u|}/2)(1-F(t-S_{j-1}))\1_{\{S_{j-1}\leq t\}}\\ \leq \exp((u^2\eee^{|u|}/2)(1-F(t-S_{j-1}))\1_{\{S_{j-1}\leq t\}}).
\end{multline*}
For the latter inequality we have used $1+x\leq \eee^x$ for $x\geq 0$. Thus, we have proved that, for $j\in\mn$, $\me_{j-1}(W_j/W_{j-1})\leq 1$ a.s. and thereupon $\me_{j-1}W_j\leq W_{j-1}$ a.s., that is, $(W_j, \mathcal{G}_j)_{j\geq 0}$ is indeed a positive supermartingale. As a consequence, the a.s.\ limit
\begin{equation*}
\lim_{j\to\infty}W_j=:W_\infty=\exp\Big(uX(t)-(u^2\eee^{|u|}/2)\sum_{k\geq 0}(1-F(t-S_k))\1_{\{S_k\leq t\}}\Big)
\end{equation*}
satisfies $\me W_\infty\leq \me W_0=1$. In other words, with $u\in\mr$ and $t>0$ fixed,
\begin{equation}\label{eq:principal}
\me\exp\Big(uX(t)-(u^2\eee^{|u|}/2)\sum_{k\geq 0}(1-F(t-S_k))\1_{\{S_k\leq t\}}\Big)\leq 1.
\end{equation}

We shall also need another auxiliary result:
\begin{equation}\label{eq:inter2}
\lim_{t\to\infty}t^{-1}\sum_{k\geq 0}(1-F(t-S_k))\1_{\{S_k\leq t\}}=0\quad\text{a.s.}
\end{equation}
To prove \eqref{eq:inter2}, write, for fixed $a>0$ and $t>a$,
\begin{multline*}
\sum_{k\geq 0}(1-F(t-S_k))\1_{\{S_k\leq t\}}=\sum_{k\geq 0}(1-F(t-S_k))\1_{\{S_k\leq t-a\}}\\+\sum_{k\geq 0}(1-F(t-S_k))\1_{\{t-a<S_k\leq t\}}\leq (1-F(a))\nu(t)+(\nu(t)-\nu(t-a)).
\end{multline*}
By the strong law of large numbers for renewal processes, $\lim_{t\to\infty}t^{-1}\nu(t)=\mu^{-1}$ a.s. and $\lim_{t\to\infty}t^{-1}(\nu(t)-\nu(t-a))=\mu^{-1}-\mu^{-1}=0$ a.s. Hence, for each fixed $a>0$, $$\limsup_{t\to\infty}t^{-1}\sum_{k\geq 0}(1-F(t-S_k))\1_{\{S_k\leq t\}}\leq \mu^{-1}(1-F(a))\quad\text{a.s.}$$ Letting $a\to\infty$ we arrive at \eqref{eq:inter2}.

Fix any $\varepsilon>0$ and put $t_n:=\exp(n^{3/4})$ for $n\in\mn$. We intend to prove that
\begin{equation}\label{eq:semifinal}
\lim_{n\to\infty}(X(t_n)/b(t_n))=0\quad\text{a.s.}
\end{equation}
To this end, for $n\geq 3$, define the event $$A_n:=\{X(t_n)>\varepsilon b(t_n)\}.$$ In view of \eqref{eq:inter2}, for large $n$, $$\sum_{k\geq 0}(1-F(t_n-S_k))\1_{\{S_k\leq t_n\}}\leq (\varepsilon^2/8)t_n.$$ Using this we obtain, for any $u>0$ and large $n$,
\begin{multline*}
A_n=\{uX(t_n)-(u^2\eee^{|u|}/2)\sum_{k\geq 0}(1-F(t_n-S_k))\1_{\{S_k\leq t_n\}}\\>\varepsilon u b(t_n)-(u^2\eee^{|u|}/2)\sum_{k\geq 0}(1-F(t_n-S_k))\1_{\{S_k\leq t_n\}}\}\\\subseteq \{uX(t_n)-(u^2\eee^{|u|}/2)\sum_{k\geq 0}(1-F(t_n-S_k))\1_{\{S_k\leq t_n\}} \\ >\varepsilon u b(t_n)-(\varepsilon^2/8)(u^2\eee^{|u|}/2)t_n\}=:B_n.
\end{multline*}
Invoking Markov's inequality in combination with \eqref{eq:principal} we infer
\begin{multline*}
\mmp(B_n)\leq \exp\Big(-\varepsilon u b(t_n)+(\varepsilon^2/8)(u^2\eee^{|u|}/2)t_n\Big)\\\times\me \exp\Big(uX(t_n)-(u^2\eee^{|u|}/2)\sum_{k\geq 0}(1-F(t_n-S_k))\1_{\{S_k\leq t_n\}}\Big)\\\leq \exp\Big(-\varepsilon u b(t_n)+(\varepsilon^2/8)(u^2\eee^{|u|}/2)t_n\Big).
\end{multline*}
Let $\rho>0$ satisfy $\exp(8\varepsilon^{-1}\rho)=3/2$. For large $x>0$, $x^{-1}\log\log x\leq \rho$. Put $$u=8\varepsilon^{-1}(t_n^{-1}\log\log t_n)^{1/2}.$$ Then $$-\varepsilon u b(t_n)+(\varepsilon^2/8)(u^2\eee^{|u|}/2)t_n\leq -8\log\log t_n+4\eee^{8\varepsilon^{-1}\rho}\log\log t_n=-2\log\log t_n.$$ Hence, by the Borel-Cantelli lemma, $\limsup_{n\to\infty}(X(t_n)/b(t_n))\leq 0$ a.s. The converse inequality for the lower limit follows analogously. Start with $A_n^\ast:=\{-X(t_n)>\varepsilon b(t_n)\}$ and show, by the same reasoning as above, that $A_n^\ast\subseteq B_n^\ast$, where $B_n^\ast$ only differs from $B_n$ by the term $-uX(t_n)$ in place of $uX(t_n)$.

It remains to show that \eqref{eq:semifinal} can be lifted to \eqref{eq:final}. To this end, it suffices to prove that
\begin{equation}\label{princip}
\lim_{n\to\infty}\frac{\sup_{u\in [t_n,\,t_{n+1}]}\,|X(u)-X(t_n)|}{t_n^{1/2}}=0\quad\text{a.s.}
\end{equation}
Indeed, \eqref{princip} in combination with \eqref{eq:semifinal} entails $$\lim_{n\to\infty}\frac{\sup_{u\in [t_n,\,t_{n+1}]}\,|X(u)|}{b(t_n)}=0\quad\text{a.s.}$$ This ensures \eqref{eq:final} because, for large enough $n$, $$\frac{|X(t)|}{b(t)}\leq \frac{\sup_{u\in [t_n,\,t_{n+1}]}\,|X(u)|}{b(t_n)}\quad\text{a.s.}$$ whenever $t\in [t_n,\,t_{n+1}]$.

We denote by $I=I_n$ a positive integer to be chosen later. For $j\in\mn_0$ and $n\in\mn$, put $$F_j(n):=\{v_{j,m}(n):=t_n+2^{-j}m(t_{n+1}-t_n):0\leq m\leq 2^j\}.$$ In what follows, we write $v_{j,\,m}$ for $v_{j,\,m}(n)$. Observe that $F_j(n)\subseteq F_{j+1}(n)$. For any $u\in [t_n, t_{n+1}]$, put
$$
u_j:=\max\{v\in F_j(n): v\leq u\}=t_n+2^{-j}(t_{n+1}-t_n)\left\lfloor\frac{2^j(u-t_n)}{t_{n+1}-t_n}\right\rfloor.
$$
An important observation is that either $u_{j-1}=u_j$ or $u_{j-1}=u_j-2^{-j}(t_{n+1}-t_n)$. Necessarily, $u_j=v_{j,m}$ for some $0\leq m\leq 2^j$, so that either $u_{j-1}=v_{j,m}$ or $u_{j-1}=v_{j,m-1}$. Write
\begin{multline*}
\sup_{u\in [t_n,\, t_{n+1}]}|X(u)-X(t_n)|\\= \max_{0\leq j\leq 2^I-1}\sup_{z\in [0,\,v_{I,\,j+1}-v_{I,\,j}]}|(X(v_{I,\,j})-X(t_n))+(X(v_{I,\,j}+z)-X(v_{I,\,j}))|\\\leq \max_{0\leq j\leq 2^I-1}|X(v_{I,\,j})-X(t_n)|\\+\max_{0\leq j\leq 2^I-1}\sup_{z\in [0,\,v_{I,\,j+1}-v_{I,\,j}]}|X(v_{I,\,j}+z)-X(v_{I,\,j})|\quad\text{a.s.} 
\end{multline*}
For $u\in F_I(n)$,
\begin{multline*}
|X(u)-X(t_n)| = \Big|\sum_{j=1}^I (X(u_j)-X(u_{j-1}))+X(u_0)-X(t_n)\Big|\\\leq \sum_{j=0}^I \max_{1\leq m\leq 2^j}\,|X(v_{j,\,m})-X(v_{j,\,m-1})|.
\end{multline*}
With this at hand, we obtain
\begin{multline}
\sup_{u\in [t_n,\, t_{n+1}]}|X(u)-X(t_n)|\leq \sum_{j=0}^I \max_{1\leq m\leq 2^j}\,|X(v_{j,\,m})-X(v_{j,\,m-1})|\\+
\max_{0\leq j\leq 2^I-1}\sup_{z\in [0,\,v_{I,\,j+1}-v_{I,\,j}]}|X(v_{I,\,j}+z)-X(v_{I,\,j})|\quad\text{a.s.}\label{eq:inter15}
\end{multline}

We first show that, for all $\varepsilon>0$,
\begin{equation}\label{eq:inter3}
\sum_{n\geq 1}\mmp\Big\{\sum_{j=0}^I \max_{1\leq m\leq 2^j}\,|X(v_{j,\,m})-X(v_{j,\,m-1})|>\varepsilon t_n^{1/2}\Big\}<\infty.
\end{equation}
Let $\ell\in\mn$. As a preparation, we derive an appropriate upper bound for $\me (X(u)-X(v))^{2\ell}$ for $u,v>0$, $u>v$. Observe that $X(u)-X(v)$ is equal to the a.s. limit $\lim_{j\to\infty} R(j,u,v)$, where $(R(j,u,v), \mathcal{G}_j)_{j\geq 0}$ is a martingale defined by $$R(0, u,v):= 0, \quad R(j, u,v):=\sum_{k=0}^{j-1}(\1_{\{v<\eta_{k+1}+S_k\leq u\}}-F(u-S_k)+F(v-S_k)),\quad j\in \mn,  $$ and, as before, $\mathcal{G}_0$ denotes the trivial $\sigma$-algebra and, for $j\in\mn$, $\mathcal{G}_j$ denotes the $\sigma$-algebra generated by $(\xi_k,\eta_k)_{1\leq k\leq j}$. Recall that $F(t)=0$ for $t<0$. By the Burkholder–Davis–Gundy inequality, see, for instance, Theorem 11.3.2 in \cite{Chow+Teicher:2003},
\begin{multline*}
\me (X(u)-X(v))^{2\ell} \leq C\Big(\me \Big(\sum_{k\geq 0}\me\big((R(k+1, u,v)-R(k, u,v))^2|\mathcal{G}_k\big)\Big)^\ell \\ +\sum_{k\geq 0}\me (R(k+1, u,v)-R(k, u,v))^{2\ell}\Big)\\=C\Big(\me \Big(\sum_{k\geq 0}(F(u-S_k)-F(v-S_k))(1-F(u-S_k)+F(v-S_k))\Big)^\ell\\+ \sum_{k\geq 0}\me(\1_{\{v<\eta_{k+1}+S_k\leq u\}}-F(u-S_k)+F(v-S_k))^{2\ell}\Big)=: C(A(u,v)+B(u,v))
\end{multline*}
for a positive constant $C$. Let $f: [0,\infty)\to[0,\infty)$ be a locally bounded function. It is shown in the proof of Lemma A.3 in \cite{Alsmeyer+Iksanov+Marynych:2017} that $\me (\nu(1))^\ell<\infty$ and that
\begin{equation}\label{eq:aux}
\me\Big(\int_{[0,\,t]}f(t-y){\rm d}\nu(y)\Big)^\ell\leq \me (\nu(1))^\ell \Big(\sum_{n=0}^{\lfloor t\rfloor} \sup_{y\in [n,\,n+1)}\,f(y)\Big)^\ell.
\end{equation}
Further,
\begin{multline*}
A(u,v)=\me \Big(\int_{(v,\,u]}F(u-y)(1-F(u-y)){\rm d}\nu(y)\\+\int_{[0,\,v]}(F(u-y)-F(v-y))(1-F(u-y)+F(v-y)){\rm d}\nu(y)\Big)^\ell\\\leq 2^{\ell-1}\Big(\me \Big(\int_{(v,\,u]}F(u-y)(1-F(u-y)){\rm d}\nu(y)\Big)^\ell\\+ \me\Big(\int_{[0,\,v]}(F(u-y)-F(v-y))(1-F(u-y)+F(v-y)){\rm d}\nu(y)\Big)^\ell\\\leq 2^{\ell-1}\Big(\me\Big(\int_{[0,\,u]}\1_{[0,\,u-v)}(u-y){\rm d}\nu(y)\Big)^\ell+\me\Big(\int_{[0,\,v]}(F(u-y)-F(v-y)){\rm d}\nu(y)\Big)^\ell\Big)\\=:2^{\ell-1}(A_1(u,v)+A_2(u,v)).
\end{multline*}
Using \eqref{eq:aux} with $t=u$ and $f(y)=\1_{[0,\,u-v)}(y)$ and then with $t=v$ and $f(y)=F(u-v+y)-F(y)$ we infer $$A_1(u,v)\leq \me (\nu(1))^\ell \Big(\sum_{n=0}^{\lfloor u\rfloor}\sup_{y\in [n,\,n+1)}\,\1_{[0,\,u-v)}(y)\Big)^\ell=\me (\nu(1))^\ell (\lceil u-v\rceil)^\ell,$$ where $x\mapsto \lceil x\rceil$ is the ceiling function, and
\begin{multline*}
A_2(u,v)\leq \me (\nu(1))^\ell \Big(\sum_{n=0}^{\lfloor v\rfloor}\sup_{y\in [n,\,n+1)}\,(F(u-v+y)-F(y))\Big)^\ell\\\leq \me (\nu(1))^\ell \Big(\sum_{n=0}^{\lfloor v\rfloor}(F(\lceil u-v\rceil+n+1)-F(n))\Big)^\ell \\ =\me (\nu(1))^\ell \Big(\sum_{n=0}^{\lceil u-v\rceil}(F(\lfloor v\rfloor+1+n)-F(n))\Big)^\ell \leq \me (\nu(1))^\ell (\lceil u-v\rceil+1)^\ell.
\end{multline*}
Finally,
\begin{multline*}
B(u,v)\leq \sum_{k\geq 0}\me(\1_{\{v<\eta_{k+1}+S_k\leq u\}}-F(u-S_k)+F(v-S_k))^2\leq 2\me \nu(1) (\lceil u-v\rceil+1)\\\leq 2\me \nu(1) (\lceil u-v\rceil+1)^\ell
\end{multline*}
and thereupon
\begin{equation}\label{eq:mom}
\me (X(u)-X(v))^{2\ell}\leq C_1(\lceil u-v\rceil+1)^\ell.
\end{equation}
Note that $v_{j,\,m}-v_{j,\,m-1}=2^{-j}(t_{n+1}-t_n)$. Given $A>0$ there exists a constant $C_2>0$ such that $C_1(\lceil 2^{-j}(t_{n+1}-t_n)\rceil+1)^\ell\leq C_2 2^{-j\ell}(t_{n+1}-t_n)^\ell$. Put $I=I_n:=\lfloor\log_2(A^{-1}(t_{n+1}-t_n))\rfloor$. Invoking \eqref{eq:mom} we then obtain, for nonnegative integer $j\leq I$,
\begin{equation}\label{eq:inter17}
\me (X(v_{j,\,m})-X(v_{j,\,m-1}))^{2\ell}\leq C_1(\lceil 2^{-j}(t_{n+1}-t_n)\rceil+1)^\ell\leq C_2 2^{-j\ell}(t_{n+1}-t_n)^\ell
\end{equation}
and thereupon
\begin{multline*}
\me \big(\max_{1\leq m\leq 2^j}\,(X(v_{j,\,m})-X(v_{j,\,m-1}))^{2\ell}\big)\leq \sum_{m=1}^{2^j}\me(X(v_{j,\,m})-X(v_{j,\,m-1}))^{2\ell}\\\leq C_2 2^{-j(\ell-1)}(t_{n+1}-t_n)^\ell.
\end{multline*}
By the triangle inequality for the $L_{2\ell}$-norm,
\begin{multline*}
\me \Big(\sum_{j=0}^I\max_{1\leq m\leq 2^j}\,|X(v_{j,\,m})-X(v_{j,\,m-1})|\Big)^{2\ell}\\\leq \Big(\sum_{j=0}^I \big(\me \big(\max_{1\leq m\leq 2^j}\,(X(v_{j,\,m})-X(v_{j,\,m-1}))^{2\ell}\big)\big)^{1/(2\ell)}\Big)^{2\ell}\\\leq C_2(t_{n+1}-t_n)^\ell\big(\sum_{j\geq 0}2^{-j(\ell-1)/(2\ell)}\big)^{2\ell}=:C_3 (t_{n+1}-t_n)^\ell.
\end{multline*}
By Markov's inequality, $$\mmp\Big\{\sum_{j=0}^I \max_{1\leq m\leq 2^j}\,|X(v_{j,\,m})-X(v_{j,\,m-1})|>\varepsilon t_n^{1/2}\Big\}\leq C_3\varepsilon^{-2\ell}t_n^{-\ell}(t_{n+1}-t_n)^\ell.$$ Since $t_n^{-1}(t_{n+1}-t_n)\sim (3/4)n^{-1/4}$ as $n\to\infty$, \eqref{eq:inter3} follows upon setting $\ell=6$, say. Invoking the Borel-Cantelli lemma we infer $$\lim_{n\to\infty}\frac{\sum_{j=0}^I \max_{1\leq m\leq 2^j}\,|X(v_{j,\,m})-X(v_{j,\,m-1})|}{t_n^{1/2}}=0\quad\text{a.s.}$$

Now we pass to the analysis of the second summand in \eqref{eq:inter15}. Put $M(t):=\int_{[0,\,t]}F(t-y){\rm d}\nu(y)$ for $t\geq 0$. Using the equality $X(t)=N(t)-M(t)$ and a.s. monotonicity of $N$ and $M$ we infer
\begin{multline*}
\sup_{z\in [0,\,v_{I,\,j+1}-v_{I,\,j}]}|X(v_{I,\,j}+z)-X(v_{I,\,j})|\\\leq \sup_{z\in [0,\,v_{I,\,j+1}-v_{I,\,j}]}(N(v_{I,\,j}+z)-N(v_{I,\,j}))\\+\sup_{z\in [0,\,v_{I,\,j+1}-v_{I,\,j}]}\,(M(v_{I,\,j}+z)-M(v_{I,\,j}))\\=(N(v_{I,\,j+1})-N(v_{I,\,j}))+(M(v_{I,\,j+1})-M(v_{I,\,j})).
\end{multline*}
Observe that
\begin{multline*}
\max_{0\leq j\leq 2^I-1}\,(N(v_{I,\,j+1})-N(v_{I,\,j}))\leq \max_{0\leq j\leq 2^I-1}\,|X(v_{I,\,j+1})-X(v_{I,\,j})|\\+\max_{0\leq j\leq 2^I-1}\,(M(v_{I,\,j+1})-M(v_{I,\,j})).
\end{multline*}
Hence, according to the Borel-Cantelli lemma, it is enough to prove that, for all $\varepsilon>0$,
\begin{equation}\label{eq:inter4}
\sum_{n\geq 1}\mmp\{\max_{0\leq j\leq 2^I-1}\,(M(v_{I,\,j+1})-M(v_{I,\,j}))>\varepsilon t_n^{1/2}\}<\infty
\end{equation}
and
\begin{equation}\label{eq:inter5}
\sum_{n\geq 1}\mmp\{\max_{0\leq j\leq 2^I-1}\,|X(v_{I,\,j+1})-X(v_{I,\,j})|>\varepsilon t_n^{1/2}\Big\}<\infty.
\end{equation}
Arguing as above we conclude that, for $u,v>0$, $u>v$,
\begin{multline*}
\me (M(u)-M(v))^\ell=\me\Big(\int_{(v,\,u]}F(u-y){\rm d}\nu(y)+\int_{[0,\,v]}(F(u-y)-F(v-y)){\rm d}\nu(y)\Big)^\ell\\\leq 2^{\ell-1}\me (\nu(1))^\ell(\lceil u-v\rceil+1)^\ell.
\end{multline*}
As a consequence, for nonnegative integer $j\leq I$ and a constant $C_4>0$, $$\me (M(v_{I,\,j+1})-M(v_{I,\,j}))^\ell\leq C_4 2^{-I\ell}(t_{n+1}-t_n)^\ell.$$ By Markov's inequality and our choice of $I$,
\begin{multline*}
\mmp\{\max_{0\leq j\leq 2^I-1}\,(M(v_{I,\,j+1})-M(v_{I,\,j}))>\varepsilon t_n^{1/2}\}\leq C_4\varepsilon^{-\ell}2^{-I(\ell-1)}t_n^{-\ell/2}(t_{n+1}-t_n)^\ell\\\leq C_4\varepsilon^{-\ell}(2A)^{\ell-1}t_n^{-\ell/2}(t_{n+1}-t_n).
\end{multline*}
Hence, \eqref{eq:inter4} follows upon choosing $\ell>2$. To prove \eqref{eq:inter5}, we invoke \eqref{eq:inter17} which enables us to conclude that
\begin{multline*}
\mmp\{\max_{0\leq j\leq 2^I-1}\,|X(v_{I,\,j+1})-X(v_{I,\,j})|>\varepsilon t_n^{1/2}\} 
\leq C_2\varepsilon^{-2\ell}2^{-I(\ell-1)}t_n^{-\ell}(t_{n+1}-t_n)^\ell\\\leq C_2\varepsilon^{-2\ell}(2A)^{\ell-1}t_n^{-\ell}(t_{n+1}-t_n).
\end{multline*}
Choosing $\ell>1$ we arrive at \eqref{eq:inter5}.

The proof of Theorem \ref{imp2} is complete.

\section{Auxiliary results}\label{sect:aux}
To prove Theorem \ref{main}, we need some auxiliary results on the iterated perturbed random walks.
Lemma \ref{lem:elem} 
is a known result, see Assertion 1 in \cite{Iksanov+Rashytov+Samoilenko:2023}.
\begin{lemma}\label{lem:elem}
Assume that $\mu = \me \xi < \infty $. Then, for fixed $j \in \mn$,
\begin{equation}\label{eq:assympt}
\lim_{t \to \infty} \frac{V_j(t)}{t^j} = \frac{1}{j!\mu^j}.
\end{equation}
\end{lemma}

Put
\begin{equation}\label{def:u}
U(t) := \me \nu(t)=\sum_{k\geq 0}\mmp{\{S_k\leq t\}} \quad \text{for } t\geq 0,
\end{equation}
so that $U$ is the renewal function.
\begin{lemma}
For every $x, h >0$ and $k \in \mn$,
\begin{equation}\label{ineq:vdiff}
V_k(x+h)-V_k(x) \leq U(h) 
(V(x+h))^{k-1}.
\end{equation}
\end{lemma}
\begin{proof} 
We use mathematical induction.
For $k = 1$, write
\begin{multline}\label{ineq:subadd}
V(x+h) - V(x)  = \int_{[0,\,x+h]} F(x-y){\rm d}U(y) - \int_{[0,\,x]} F(x-y){\rm d}U(y)\\
 = \int_{(x,\, x+h]} F(x-y){\rm d}U(y) \leq U(x+h) - U(x) \leq U(h).
\end{multline}
The last inequality is due to subadditivity of the renewal function $U$, see Theorem 1.7 on p.10 in  \cite{Mitov+Omey:2014}.

Assume that inequality 
\eqref{ineq:vdiff} holds for $k \leq l-1$. Note that \eqref{eq:vdecomp} implies that $V_{l-1}(h) \leq (V(h))^{l-1} \leq U(h)(V(x+ h))^{l-2}$ for $l\geq 2$ and $h\geq 0$. 
Using this and the induction assumption, we have
\begin{multline*}
V_l(x+h)-V_l(x)  = \int_{[0,\,x]} (V_{l-1}(x+ h-y) - V_{l-1}(x-y)) {\rm d} V(y) \\ + \int_{(x, x+h] } V_{l-1}(x+ h-y){\rm d} V(y) \\
\leq U(h) \int_{[0,\,x]} (V(x+h-y))^{l-2} {\rm d} V(y) + V_{l-1}(h) (V(x+h)-V(x)) \\
\leq  U(h) (V(x+h))^{l-2} \cdot V(x) + U(h) (V(x+h))^{l-2} (V(x+h)-V(x))\\ =  U(h) (V(x+h))^{l-1}.
\end{multline*}
\end{proof}
\begin{lemma}\label{lem:var}
Assume that $ {\rm Var}\,\xi\in (0,\infty)$. Then, for $k \in \mn$,
\begin{equation}\label{eq:var}
a_k(t):=  {\rm Var}\, Y_k(t) = O(t^{2k-1}), \quad t \to \infty.
\end{equation}
\end{lemma}
\begin{proof}
We use mathematical induction. For $k=1$, write
\begin{multline*}
Y_1(t) - V_1(t) = \sum_{j\geq 1}(\1_{\{ \eta_j + S_{j-1}\leq t \}}-F(t-S_{j-1})) + ( \sum_{j\geq 0}F(t-S_j)-V_1(t))\\ =:I_1(t)+I_2(t).
\end{multline*}
Note that ${\rm Var}\,Y_1(t) = \me (Y_1(t) - V_1(t))^2 \leq 2(\me(I_1(t))^2 + \me(I_2(t))^2 ) $. Let $U$ be as in \eqref{def:u}. We have
$$
\me(I_1(t))^2 = \int_{[0,\,t]} F(t-y)(1-F(t-y)){\rm d}U(y) \leq \int_{[0,\,t]}(1-F(t-y)){\rm d}U(y).
$$
If $\me \eta = \infty$, then Lemma 6.2.9 in \cite{Iksanov:2016} with $r_1=0$ and $r_2=1$ yields
$$
\int_{[0,\,t]} (1-F(t-y)){\rm d}U(y) \sim \frac{1}{\mu} \int_0^t (1-F(y)) 
{\rm d}y = o(t),  \quad  t \to \infty,
$$
where $\mu=\me \xi<\infty$. If $\me \eta < \infty$, then $\int_{[0,\,t]} (1-F(t-y)){\rm d}U(y) = O(1)$ as $t \to \infty$ by the key renewal theorem. Thus, in any case, $\me(I_1(t))^2 = o(t)$ as $ \ t \to  \infty$.

In the proof of Lemma 4.2 in \cite{Gnedin+Iksanov:2020} it is shown that
\begin{equation}\label{eq:standard}
\me \sup_{s \in[0,\,t]} (\nu(s)-U(s))^2 = O(t), \quad t \to \infty,
\end{equation}
where $\nu(s)$ is the same as in \eqref{def:nu} (one of the assumptions in Lemma 4.2 of \cite{Gnedin+Iksanov:2020} is $\me \eta< 
\infty$; but this condition is not needed for \eqref{eq:standard}, because the left-hand side of \eqref{eq:standard} is a functional depending on $(S_j)_{j\geq 0}$ only). Therefore,  almost surely
\begin{multline}\label{eq:i2}
|I_2(t)| = \Big| \int_{[0,\,t]}F(t-y) {\rm d} (\nu(y) - U(y))  \Big| = \Big|\int_{[0,\,t]} (\nu(t-y) - U(t-y){\rm d} F(y) \Big|  \leq \\
\leq \int_{[0,\,t]} |\nu(t-y) - U(t-y)|{\rm d} F(y) \leq \sup_{s \in[0,\,t]} |\nu(s)-U(s)|\cdot F(t) \leq \sup_{s \in[0,\,t]} |\nu(s)-U(s)|.
\end{multline}
Consequently,  according to \eqref{eq:standard},  $\me (I_2(t))^2 \leq \me \sup_{s \in[0,\,t]} (\nu(s)-U(s))^2 = O(t)$ as $t \to \infty$. We have proved that $a_1(t) = O(t)$ as $ t \to \infty$.

Assume that relation 
\eqref{eq:var} holds for $k \leq l-1$. We shall use a representation 
$$
Y_l(t) - V_l(t) = \sum_{r \geq 1}\big(Y_{l-1}^{(r)}(t - T_r) - V_{l-1}(t-T_r)\big) + \big( \sum_{r\geq1} V_l(t-T_r) - V_l(t) \big) =: J_l(t) + K_l(t),
$$
which particularly entails 
$$
a_l(t) = \me (Y_l(t) - V_l(t))^2 = \me (J_l(t))^2 + \me(K_l(t))^2.
$$
Note that, according to the induction assumption, there exist $A>0$ and $t_0>0$ such that $a_{l-1}(t) \leq A t^{2l-3}$ for all $t \geq t_0$. Therefore, using \eqref{eq:assympt}  and  \eqref{ineq:subadd},
\begin{multline}\label{eq:jlimit}
\me (J_l(t))^2 = \int_{[0,\,t]} a_{l-1}(t-y){\rm d}V(y)\\ = \int_{[0,\,t-t_0]} a_{l-1}(t-y){\rm d}V(y) + \int_{(t-t_0,\,t]} a_{l-1}(t-y){\rm d}V(y)\\
\leq A \int_{[0,\,t]}(t-y)^{2l-3} {\rm d}V(y) + \sup_{s \in [0,t_0]} a_{l-1}(s) (V(t) - V(t-t_0))  \leq At^{2l-3}V(t) + O(1) \\
= O(t^{2l-2}), ~~ t \to \infty.
\end{multline}
Further,
\begin{multline*}
K_l(t) = \sum_{r\geq 0} \big( V_{l-1}(t-T_r) - (V_{l-1}\ast F)(t-S_{r-1}) \big) +   \big( \sum_{r\geq 0} (V_{l-1}\ast F)(t-S_{r-1}) - V_l(t)) \big) \\
=: K_{l1}(t) + K_{l2}(t).
\end{multline*}
Using $V_l = V_{l-1}\ast F\ast U$ and the same reasoning as in \eqref{eq:i2} we obtain
$$
|K_{l2}(t)| = \Big|  \int_{[0,\,t]} (V_{l-1}\ast F)(t-y) {\rm d} (\nu(y) - U(y))\Big| \leq \sup_{s \in [0,\,t]} |\nu(s)- U(s)|\cdot V_{l-1}(t) \ \text{ a.s.}
$$
Therefore, in view of \eqref{eq:assympt} and \eqref{eq:standard}, $$\me (K_{l2}(t))^2 \leq \me \sup_{s\in [0,\,t]}(\nu(s)- U(s))^2 (V_{l-1}(t))^2 = O(t^{2l-1}),\quad t \to \infty.$$ 

Finally,
\begin{multline*}
\me (K_{l1}(t))^2 = \sum_{r\geq 1} \me \big( V_{l-1}(t- T_r) - (V_{l-1}\ast F)(t-S_{r-1})  \big)^2   \\
\leq \sum_{r\geq 1} \Big[ \me \big(V_{l-1}(t- T_r)\big)^2 + \me \big((V_{l-1}\ast F)(t-S_{r-1})\big)^2 \Big] \\
= \int_{[0,\,t]} \big(V_{l-1}(t-y)\big)^2 {\rm d}V(y) + \int_{[0,\,t]} \big((V_{l-1}\ast F)(t-y)\big)^2 {\rm d}U(y)\\
\leq \big( V_{l-1}(t)\big)^2\cdot V(t) + \big(( V_{l-1}\ast F)(t)\big)^2\cdot U(t) = O(t^{2l-1}), \quad  t \to \infty.
\end{multline*}
For the last equality we have used $(V_{l-1}\ast F) (t) \leq V_{l-1}(t)$ for $t\geq 0$. The proof of the Lemma \ref{lem:var} is complete.
\end{proof}

We shall also need two results on the standard random walks.
The next lemma is a consequence of formula (33) in \cite{Alsmeyer+Iksanov+Marynych:2017}, with $\eta = \xi$.

\begin{lemma}\label{lem:aim}
For all positive b and c,
$$
\lim_{t\to \infty} \frac{\nu(t+b) - \nu(t)}{t^c} = 0 \quad {\rm  a.s.}
$$
\end{lemma}\label{lem:diff}

\begin{lemma}\label{lem:aux}
Let $K_1, \ K_2\ : \ [0,\infty) \to [0,\infty)$ be nondecreasing functions and $K_1(t) \geq K_2(t)$ for $t \geq 0$. Assume that
\begin{equation}\label{assump:lemaux}
\limsup_{t\to \infty} \frac{K_1(t) + K_2(t)}{\int_0^t (K_1(y) - K_2(y)){\rm d}y} \leq \lambda \in (0, \infty).
\end{equation}
Then, for all $c>0$,
\begin{equation}\label{assert:lemaux}
\lim_{t \to \infty} \frac{\int_{[0,\,t]} (K_1(t-y) - K_2(t-y)){\rm d}\nu(y)}{ t^c \int_0^t (K_1(y) - K_2(y)){\rm d}y} = 0 \quad\text{{\rm a.s.}}
\end{equation}
\end{lemma}
\begin{proof}
We use a decomposition
$$
\int_{[0,\,t]} (K_1(t-y) - K_2(t-y)){\rm d}\nu(y) = \int_{[0,\,[t]]}... + \int_{[[t],\, t]}... =: I_1(t) + I_2(t).
$$
For $I_2(t)$ we have
$$
I_2(t)\leq \int_{[[t],\, t]} K_1(t-y) {\rm d}\nu(y) \leq K_1(t-[t])(\nu(t) - \nu([t]))
\leq K_1(1)(\nu(t) - \nu(t-1)).
$$
Hence, by Lemma \ref{lem:aim}, for all $c>0$, 
$\lim_{t \to \infty} t^{-c} I_2(t)=0$ a.s.
It remains to consider $I_1(t)$:
\begin{multline*}
I_1(t) = K_1(t) - K_2(t) + \sum_{j=0}^{[t]-1} \int_{(j,\,j+1]} (K_1(t-y) - K_2(t-y)){\rm d}\nu(y) \\
\leq K_1(t) - K_2(t)+ \sum_{j=0}^{[t]-1} (K_1(t-j) - K_2(t-j-1)) (\nu(j+1) - \nu(j))\\
\leq K_1(t) + \sup_{s \in[0,\,[t]]} (\nu(s+1) - \nu(s)) \sum_{j=0}^{[t]-1} (K_1(t-j) - K_2(t-j-1))\\
\leq K_1(t) + \sup_{s \in[0,[t]]} (\nu(s+1) - \nu(s)) \sum_{j=0}^{[t]-1} (K_1([t]+1-j) - K_2([t]-1-j)) \\
= \sup_{s \in[0,[t]]} (\nu(s+1) - \nu(s)) \left( \int_2^{[t]} 
(K_1(y) - K_2(y)){\rm d} y   +O(K_1(t) + K_2(t))\right).
\end{multline*}
Another application of Lemma \ref{lem:aim} 
yields
$$\lim_{t\to \infty} \frac{I_1(t)}{t^c \int_0^t (K_1(y) - K_2(y)){\rm d}y} = 0 \quad \text{ a.s.}
$$
The proof of Lemma \ref{lem:aux} is complete.
\end{proof}

\section{Proof of Theorem \ref{main}}

We use a decomposition
\begin{equation}\label{eq:decomp}
Y_j(t) - V_j(t)  = \sum_{k\geq 1} \big( Y_{j-1}^{(k)} (t-T_k) - V_{j-1} (t-T_k) \big) +  \sum_{k\geq 1}  V_{j-1} (t-T_k) - V_j(t),\quad j\geq 2,~t\geq 0.
\end{equation}
The first term of the decomposition is treated in Proposition \ref{second}. 
\begin{proposition}\label{second}
Assume that ${\rm Var}\,\xi\in (0,\infty)$. Then, for $j\geq 2$,
$$
\lim_{t\to \infty}\frac{ \sum_{k\geq 1} \big( Y_{j-1}^{(k)} (t-T_k) - V_{j-1} (t-T_k) \big)}{( t^{2j-1}\log\log t)^{1/2} } = 0 \quad\text{{\rm a.s.}}
$$
\end{proposition}

We first prove Theorem \ref{main} with the help of Proposition \ref{second}. Afterwards, a proof of Proposition \ref{second} will be given.


\begin{proof}[Proof of Theorem \ref{main}]
By Proposition \ref{second}, the contribution of the first term in \eqref{eq:decomp} normalized by $(t^{2j-1} \log \log t)^{1/2}$ vanishes as $t \to \infty$.

For the second term in \eqref{eq:decomp}, write
\begin{multline*}
\sum_{k\geq 1}  V_{j-1} (t-T_k) - V_j(t) = \int_{[0,\,t]} Y(t-x) {\rm d} V_{j-1}(x) - V_j(t)\\
= \int_{[0,\,t]} \big( Y(t-x) - (F \ast \nu)(t-x) \big) {\rm d} V_{j-1}(x) + \Big( \int_{[0,\,t]} (F \ast \nu)(t-x) {\rm d} V_{j-1}(x)  - V_j(t) \Big) \\=: A_1(t) + A_2(t).
\end{multline*}
According to \eqref{eq:final}, $\lim_{t\to \infty}\frac{ Y(t) - (F \ast \nu)(t)}{ ( t\log \log t)^{1/2}} = 0$ a.s., whence 
$$\lim_{t\to \infty} \sup_{z \in [0,\,t]} \frac{ |Y(z) - (F \ast \nu)(z)| }{ ( t\log \log t)^{1/2}} = 0 \quad \text{a.s.} $$
With this at hand,
\begin{multline*}
\frac{|A_1(t)|}{t^{j-1/2} (\log \log t)^{1/2}} \leq  \sup_{z \in [0,\,t]} \frac{ |Y(z) - (F \ast \nu)(z)| }{ ( t\log \log t)^{1/2}} \cdot \frac{V_{j-1}(t)}{t^{j-1}} \longrightarrow 0 \quad \text{a.s.,} \quad  t \to \infty, \\
\text{using } \frac{V_{j-1}(t)}{t^{j-1}}   \longrightarrow \frac{1}{(j-1)! \mu^{j-1}}, \quad  t \to \infty.
\end{multline*}
Further, 
\begin{multline*}
A_2(t) = (F \ast \nu \ast V_{j-1})(t) - V_j(t) = \int_{[0,\,t]} (F\ast V_{j-1}) (t-x){\rm d}(\nu(x) -U(x))\\= \int_{[0,\,t]} (\nu(t-x) -U(t-x)){\rm d}(F\ast V_{j-1}) (x).
\end{multline*}
Recall that the distribution of $\eta$ is arbitrary. 
Now we show that in the subsequent proof $F$ can be replaced with an absolutely continuous distribution function that has a directly Riemann integrable (dRi) density.

Put $G(x) := 1 - \eee^{-x}$ for $x\geq0$. The function $H := F \ast G$ is absolutely continuous with density $h(x) = \int_{[0,\,x]} \eee^{-(x-y)} {\rm d} F(y)$ for $x\geq 0$. Since $x \mapsto \eee^{-x}$ is dRi on $[0,\infty)$, so is $h$ 
as a  Lebesgue–Stieltjes convolution of a dRi function and a distribution function, see Lemma 6.2.1 (c) in \cite{Iksanov:2016}. Note that $H(x) \leq F(x)$ for $x \geq 0$. To show that we can work with $H$ instead of $F$, it suffices to check that
\begin{equation}\label{eq:replace1}
\lim_{t\to \infty} \frac{(F \ast V_{j-1} \ast \nu) (t) - (H \ast V_{j-1} \ast \nu) (t)}{t^{j-1/2}} = 0 \quad\text{{\rm a.s.}}
\end{equation}
and
\begin{equation}\label{eq:replace2}
\lim_{t\to \infty} \frac{(F \ast V_{j-1} \ast U) (t) - (H \ast V_{j-1} \ast U) (t)}{t^{j-1/2}} = 0. 
\end{equation}
For \eqref{eq:replace2}, write
\begin{multline*}
(F \ast V_{j-1} \ast U) (t) - (H \ast V_{j-1} \ast U) (t) = \int_{[0,\,t]} (1-G(t-x)) {\rm d} V_j(x) \\ = \int_{[0,\,t]} \eee^{-(t-x)} {\rm d} V_j(x)
\sim \Big(\int_0^\infty \eee^{-y} {\rm d}y\Big) V_{j-1}(t) = V_{j-1}(t) \sim \frac{t^{j-1}}{(j-1)! \mu^{j-1}}, \quad t \to \infty,
\end{multline*}
where the asymptotic equalities are justified by  \eqref{eq:assympt} and Theorem 2 in \cite{Iksanov+Rashytov+Samoilenko:2023}. 
This proves \eqref{eq:replace2}.

To prove \eqref{eq:replace1}, we use Lemma \ref{lem:aux} with $K_1(t)  = (F \ast V_{j-1})(t)$ and $K_2(t)  = (H \ast V_{j-1})(t)$ for $t \geq 0$. Note that $K_2(t)=\me K_1(t-\theta)\1_{\{\theta \leq t \}}$, where $\theta$ is a random variable with the distribution function $G$, and that
\begin{multline*}
0 \leq \int_0^t (K_1(y) - K_2(y)) {\rm d}y = \int_0^t (K_1(y) - \me K_1(y-\theta) \1_{\{\theta \leq y \}})  {\rm d}y\\= \int_0^t K_1(y) {\rm d}y \cdot \eee^{-t} + \me \int_{t-\theta}^t K_1(y) {\rm d}y \1_{\{\theta \leq t \}}.
\end{multline*}
Using Laplace transforms and \eqref{eq:assympt}, we have
$$
K_1(t) \sim K_2(t) \sim V_{j-1}(t) \sim \frac{t^{j-1}}{(j-1)! \mu^{j-1}},\quad t \to \infty.
$$
Therefore,
$$
\lim_{t \to \infty}\frac{\int_0^t K_1(y) {\rm d}y \cdot \eee^{-t}}{K_1(t)} = 0
$$
and, in view of monotonicity,
$$
\frac{\me \theta K_1(t-\theta)\1_{\{\theta \leq t\}}}{K_1(t)} \leq \frac{\me \int_{t-\theta}^t K_1(y) {\rm d}y \1_{\{\theta \leq t\}}}{K_1(t)} \leq \me \theta = 1.
$$
By Lebesgue's  dominated convergence theorem
$$
\int_0^t (K_1(t) - K_2(t)) {\rm d}y \sim K_1(t) \sim \frac{t^{j-1}}{(j-1)! \mu^{j-1}},\quad t \to \infty.
$$
Thus, condition \eqref{assump:lemaux} holds with $\lambda = 2$. Consequently, \eqref{eq:replace1} holds by \eqref{assert:lemaux} with $c=1/2$. 

As a consequence of \eqref{eq:replace1} and \eqref{eq:replace2}, we can and do investigate $\hat A_2(t)= \int_{[0,\,t]} (\nu(t-x) -U(t-x)){\rm d}(H\ast V_{j-1}) (x)$ in place of $A_2(t)$. 
By 
Lemma 3.1 in \cite{Iksanov+Kabluchko+Kotelnikova:2025}, there exists a standard 
Brownian motion $(W(t))_{t\geq 0}$ such that
\begin{equation}\label{eq:brownian}
\lim_{t \to \infty}\frac{\sup_{z \in [0,\,t]} |\nu(z) -U(z) - \sigma \mu^{-3/2} W(z)| }{(t \log \log t)^{1/2}} = 0 \quad \text{a.s.}
\end{equation}
With this specific $(W(t))_{t\geq 0}$, write
\begin{multline*}
\hat A_2(t) =  \int_{[0,\,t]} \Big( \nu(t-x) -U(t-x) - \sigma \mu^{-3/2} W(t-x) \Big)  {\rm d} (H\ast V_{j-1}) (x) \\ +  \sigma \mu^{-3/2} \int_{[0,\,t]} W(t-x)   {\rm d} (H\ast V_{j-1}) (x)
=: B_1(t) + \sigma \mu^{-3/2} B_2(t).
\end{multline*}
Then, using 
\eqref{eq:brownian} and \eqref{eq:assympt}, we have
\begin{multline*}
|B_1(t)|  \leq \sup_{z \in [0,\,t]} |\nu(z) -U(z) - \sigma \mu^{-3/2} W(z)| \cdot( H\ast V_{j-1})(t) \\
\leq \sup_{z \in [0,\,t]} |\nu(z) -U(z) - \sigma \mu^{-3/2} W(z)| \cdot V_{j-1}(t) = o((t^{2j-1}\log \log t)^{1/2}), \quad  t \to \infty.
\end{multline*}

We are left with showing that
$$
C\bigg(\bigg( \frac{ (j-1)! \mu^{j-1} B_2(t)}{( 2(2j-1)^{-1} t^{2j-1}\log\log t)^{1/2}} \,:\, t> \eee\bigg)\bigg)  = [-1,1] \quad\text{{\rm a.s.}}
$$
Since $H$ is absolutely continuous with a 
dRi density, the function $ H \ast V_{j-1}$ is almost everywhere differentiable with
$$
( H \ast V_{j-1})^\prime (x) = \int_{[0,\,x]} h(x-y) {\rm d}V_{j-1}(y) \quad \text{for almost every }x\geq0.
$$
Consequently,
$$
\int_{[0,\,t]} W(t-x)   {\rm d} (H \ast V_{j-1}) (x)  = \int_{[0,\,t]} W(t-x) ( H \ast V_{j-1})'(x){\rm d}x.
$$
By Theorem 2 in \cite{Iksanov+Rashytov+Samoilenko:2023}, for 
$j\geq 2$,
\begin{equation}\label{eq:sim}
\int_{[0,\,x]} h(x-y) {\rm d}V_{j-1}(y) \sim \int _0^\infty h(y)  {\rm d} y \cdot \frac{x^{j-2}}{ (j-2)! \mu^{j-1}}= 
\frac{x^{j-2}}{ (j-2)! \mu^{j-1}},\quad x \to \infty.
\end{equation}
Іn particular, $( H \ast V_{j-1})^\prime$ varies regularly at infinity with index $j-2$, and Proposition 2.4 in \cite{Iksanov+Jedidi:2024} yields
$$
C\bigg(\bigg(\frac{(j-1)!\mu^{j-1}}{t^{j-1}}\frac{ \int_{[0,\,t]} W(t-x) ( H \ast V_{j-1})^\prime (x){\rm d}x  }{( 2(2j-1)^{-1}  t\log\log t)^{1/2} 
} \,:\, t> \eee \bigg)\bigg)  = [-1,1] \quad\text{{\rm a.s.}}
$$
Here, we have used that \eqref{eq:sim} entails $$( H \ast V_{j-1})(x) \sim  \frac{x^{j-1}}{(j-1)! \mu^{j-1}},\quad x \to \infty,$$
see Proposition 1.5.8 in \cite{Bingham:1987}. The proof of Theorem \ref{main} is complete.

\end{proof}


Finally, we prove Proposition \ref{second}. 
\begin{proof}[Proof of Proposition \ref{second}]
Put $Z_j(t) = \sum_{k\geq 1} \big( Y_{j-1}^{(k)} (t-T_k) - V_{j-1} (t-T_k) \big) $ for $t\geq 0$.
Relation 
\eqref{eq:var} implies that there exist $t_0>0$ and $A>0$ such that $a_{j-1}(t) \leq A t^{2j-3}$ for all $t \geq t_0$. Using the same reasoning as in \eqref{eq:jlimit}, we have
\begin{equation}\label{eq:limit2}
\me ( Z_j(t) )^2 = \int_{[0,\,t]} a_{j-1} (t-x) {\rm d} V(x) = O(t^{2j-2}), \quad t \to \infty.
\end{equation}
By Markov's inequality and \eqref{eq:limit2}, for all $\varepsilon >0$,
$$
\sum_{n\geq 1} \mmp \left\lbrace \frac{|Z_j(n^{3/2})|}{n^{(3/2)(j-1/2)}} > \varepsilon \right\rbrace \leq \sum_{n\geq 1} \frac{\me (Z_j(n^{3/2}))^2}{\varepsilon^2 n^{3(j-1/2)}} < \infty.
$$
Hence, by the Borel–Cantelli lemma,
\begin{equation}\label{eq:limit3}
\lim_{n \to \infty} \frac{Z_j(n^{3/2})}{n^{(3/2)(j-1/2)}} = 0 \quad \text{a.s.}
\end{equation}
It remains to pass from an integer argument to a continuous argument. For any $t \geq 0$ there exists $n \in \mn_0$ such that $t \in [n^{3/2}, (n + 1)^{3/2})$. By monotonicity,
\begin{multline*}
\frac{Z_j(t)}{t^{j-1/2}} \leq \frac{Z_j((n+1)^{3/2})}{n^{(3/2)(j-1/2)}} \\ + \frac{\int_{[0,\,(n+1)^{3/2}]}V_{j-1} ((n+1)^{3/2} - x) {\rm d} Y(x) - \int_{[0,\,n^{3/2}]} V_{j-1} (n^{3/2} - x){\rm d} Y(x)}{n^{(3/2)(j-1/2)}}.
\end{multline*}
Relation \eqref{eq:limit3} implies that the first summand on the right-hand side converges to $0$ a.s. \ as $n \to \infty$. The second summand is equal to
\begin{multline*}
\int_{(n^{3/2},(n+1)^{3/2}]}V_{j-1} ((n+1)^{3/2} - x) {\rm d} Y(x)   \\
+ \int_{[0,n^{3/2}]} (V_{j-1} ((n+1)^{3/2} - x) - V_{j-1}(n^{3/2} - x)) {\rm d} Y(x) =: X_{j,1}(n) + X_{j, 2} (n).
\end{multline*}
By monotonicity, for $j\geq 2$, as $n \to \infty$, a.s.
\begin{multline*}
X_{j,1}(n) \leq V_{j-1}((n+1)^{3/2} - n^{3/2}) (Y((n+1)^{3/2}) - Y(n^{3/2}))\\ = O(n^{j/2 + 1})  = o(n^{(3/2)(j-1/2)}).
\end{multline*}
Here, the penultimate equality is justified by the inequality $ Y(t) \leq \nu(t)$ for $t\geq 0$, the strong law of large numbers for renewal processes
$\lim_{n\to\infty} n^{-1} \nu (n) = \mu^{-1}$ a.s. and $V_{j-1} ((n+1)^{3/2} - n^{3/2}) = O(n^{(j-1)/2})$ as $n \to \infty$, which holds true by \eqref{eq:assympt}.

Using \eqref{ineq:vdiff}, we infer
\begin{multline*}
X_{j,2}(n) \leq U 
((n+1)^{3/2} - n^{3/2}) (V((n+1)^{3/2}))^{j-2} Y(n^{3/2})\\ = O(n^{(3/2)(j-2/3)}) = o(n^{(3/2)(j-1/2)})
\end{multline*}
a.s. as $n \to \infty$. The penultimate equality is secured by the elementary renewal theorem, 
the strong law of large numbers for renewal processes and the inequality $ Y(t) \leq \nu(t)$ for $t\geq 0$.

We have shown that
$$
\limsup_{t \to \infty} t^{-(j-1/2)} Z_j(t) \leq 0 \quad \text{a.s.}
$$
An analogous argument proves the converse inequality for the lower limit. The proof of Proposition \ref{second} is complete.
\end{proof}

\begin{acknowledgement}[title={Acknowledgment}]
I gratefully acknowledge the support of my academic supervisor Alexander Iksanov. His useful advice and suggestions were really helpful for me during this investigation.
\end{acknowledgement}

\bibliographystyle{bib/vmsta-mathphys}
\bibliography{bib/paper}

\end{document}